\documentclass[final]{siamltex}
\usepackage{amsmath}
\usepackage{amssymb}
\usepackage{bbm}
\usepackage{pgfplots,tikz}
\usepackage{algorithm}
\usepackage{booktabs}

\newcommand{\minimize}[2]{\ensuremath{\underset{\substack{{#1}}}%
{\mathrm{minimize}}\;\;#2 }}
\newcommand{\Argmin}[2]{\ensuremath{\underset{\substack{{#1}}}%
{\mathrm{Argmin}}\;#2 }}
\newcommand{\ran}{\ensuremath{\operatorname{ran}}}
 
\newcommand{\menge}[2]{\big\{{#1}~\big |~{#2}\big\}}
\newcommand{\NN}{\ensuremath{\mathbb N}}
\newcommand{\RR}{\ensuremath{\mathbb R}}

\newcommand{\RPP}{\ensuremath{\left]0,+\infty\right[}}
\newcommand{\sumover}[1]{\sum\limits_{i\in #1}}

\newcommand{\argmin}{\mathop{\rm argmin}}

\newcommand{\sign}{\ensuremath{\operatorname{sign}}}
\newcommand{\dom}{\ensuremath{\operatorname{dom}}}
\newcommand{\RX}{\ensuremath{\,\left]-\infty,+\infty\right]}}
\newcommand{\emp}{\ensuremath{{\varnothing}}}
\newcommand{\pinf}{\ensuremath{+\infty}}
\newcommand{\prox}{\ensuremath{\operatorname{prox}}}
\newcommand{\pto}[2]{\langle #1,#2 \rangle}

\newtheorem{remark}[theorem]{Remark}

\title{ Regularization of $\ell_1$ minimization for dealing with outliers and noise 
in  Statistics and Signal Recovery}

\author{Salvador Flores\thanks{Centro de Modelamiento Matem\'atico (CNRS UMI 2807),
Universidad de Chile. Supported by CONICYT under grant FONDECYT N$^\text{o}$
3120166.} \and
Luis M. Brice\~no-Arias\thanks{Universidad T\'ecnica Federico Santa Mar\'ia--
 Departamento de Matem\'atica. This work was supported by CONICYT
under grants FONDECYT N$^\text{o}$ 3120054 and Anillo ACT1106, by
``Programa de financiamiento basal'' from CMM--Universidad de
Chile, and by Project MathAm Sud N$^\text{o}$ 13MATH01.} 
 }
\begin{document}

\maketitle

\begin{abstract}
We study the robustness properties of $\ell_1$ norm minimization for the classical linear regression problem  
with a given design matrix and contamination restricted to the dependent variable. We perform a fine error analysis 
of the $\ell_1$ estimator for measurements errors consisting of outliers coupled with noise.
We introduce a new estimation technique resulting from a regularization of $\ell_1$ minimization by inf-convolution with the 
$\ell_2$ norm. Concerning robustness to large outliers, the proposed estimator keeps the breakdown point of the $\ell_1$ estimator, 
and reduces to least squares when there are not outliers.
We present a globally convergent forward-backward algorithm for computing our estimator and some numerical experiments confirming its
theoretical properties.
\end{abstract}



\pagestyle{myheadings}
\thispagestyle{plain}
\markboth{S. FLORES AND L.~M. BRICE\~NO-ARIAS}{Regularization of $\ell_1$ minimization}

\section{Introduction}
In this paper we adress the problem of recovering a vector $f\in \RR^p$ from 
a set of $n$ measurements ($p<n$), 
\begin{equation}\label{e:mainmodel}
y=Xf+\delta, 
\end{equation}
where $y\in\RR^n$ is the vector of measurements or observations, $X$
is an $n\times p$ matrix of full rank, whose rows are realizations of
the explicative variables, and $\delta\in\RR^n$ is an error term. 

In classical linear regression, a vector of observations
or dependent variables $y\in\RR^n$ is given along with the same
number of explicative variables $x_1,...,x_n \in \RR^p$. 
We assume that the random variables $x_1,...,x_n$ and $y$ are related through a linear model, which implies the 
existence of a vector $f\in \RR^p$ such that
\begin{equation}\label{E:linmod}
(\forall i\in \{1,...,n\})\quad y_i = x_{i}^\top f +\delta_i,
\end{equation}
where $(\delta_i)_{1\leq i\leq n}$ are i.i.d. random variables independent of the $x_i$s with zero mean and finite variance. 
The objective in linear regression is to estimate $f$.
The Least Squares Estimator (LSE) of $f$ is defined as the solution to
\begin{equation}\label{P:LS}
\begin{array}{cc}
 \min\limits_{g\in\RR^p,r\in \RR^n} & \sum\limits_{i=1}^n r_i^2 \\ 
s.t\ &  r=y-Xg,
\end{array}
\end{equation}
where $r$ denotes the vector of residuals. Under the common assumption that the  errors $\delta_i$ are gaussian,
the LSE is the best linear unbiased estimator of $f$ \cite{Shao}. However, it
is very sensitive to deviations from normality, even moderated ones.
As the hypothesis of normality is often violated in practice, there is
a great interest in developing statistical procedures that are robust
face to different error distributions. 

In robust regression, model \eqref{E:linmod} is enlarged  by considering 
that errors come from \emph{contaminated distributions} \cite{Tukey60}  
\begin{equation}\label{E:contaminated}
\delta \sim (1-\varepsilon)F_I+\varepsilon F_{II},
\end{equation}
where $F_I$ is a light-tail distribution, usually normal, and
$F_{II}$ is an arbitrary distribution, supposed to model outliers. The
quantity $0<\varepsilon<1$ represents the fraction of contamination.  
The ability of an estimation method to give reasonable results under
model \eqref{E:contaminated} can be measured by the
\emph{Regression Breakdown Point} (RBP), defined  as the maximum
fraction of the components of $\delta$ that
can diverge while keeping the estimator bounded. The LSE 
has an asymptotic RBP of $0\%$, since a single divergent
observation can completely mislead the fit, independently of the
sample size. There exists many robust estimators with the highest possible RBP (see \cite{Rousseeuw-Leroy87,MMY} for a comprehensive 
treatment of the subject),  
but all of them involve solving hard global and/or combinatorial optimization problems. 
The M-estimator \cite{Huber73,Huber81} is the first attempt to obtain robust and efficiently
computable estimators. They are a
generalization of \eqref{P:LS}, defined as a solution to
\begin{equation}\label{P:M-estimators}
\begin{array}{cc}
 \min\limits_{g\in\RR^p,r\in \RR^n} &\sum\limits_{i=1}^n \rho(r_i) \\ 
s.t\ &  r=y-Xg,
\end{array}
\end{equation}
for some differentiable pair function $\rho:\RR\to \RR_+$ which is non-decreasing in $\RR_+$. 
The first order optimality conditions of problem  \eqref{P:M-estimators} yields
\begin{equation}\label{E:M-estimators-weights}
\sum\limits_{i=1}^n w_i x_i=0,
\end{equation}
where $w_i:= \rho'(r_i)$ acts as a weight of the influence of each
observation on the fit. Hence, if the function $\rho$ is additionally convex the observations with large residuals have a higher weight. 
This implies that the M-estimator is sensitive to outliers in this case. 
In the opposite case, if the function $\rho$ has non-increasing derivative, we face a nonconvex optimization problem, which are beyond the 
capabilities of the state-of-the-art of optimization methods, even for problems of modest size.

The border case is the $\ell_1$ estimator, also called Least Absolute Deviations, which is defined as a solution to 
\begin{equation}\label{P:l1}
\begin{array}{cc}
 \min\limits_{g\in\RR^p,r\in \RR^n} & \sum\limits_{i=1}^n  |r_i| \\ 
s.t\ &  r=y-Xg.
\end{array}
\end{equation}
It does not fit in the framework of \eqref{P:M-estimators} since the
function to minimize is not differentiable. Nonetheless, it satisfies
equation \eqref{E:M-estimators-weights} for $w_i$ equal to one if
$r_i>0$, equal to minus one if $r_i<0$, and between $-1$ and $1$ for
null residuals. Therefore, the $\ell_1$ estimator gives a bounded weight to each observation while keeping the 
estimation problem convex.

Despite of the remarkable properties of the $\ell_1$ estimator, it has been difficult to find its place in robust regression. 
In fact, most of the literature on the subject  adopts the  notion of breakdown point of Donoho-Huber
\cite{Donoho-Huber}, that considers the effect of replacing a subset
of \emph{pairs} $(x_i,y_i)$ of observations by arbitrary ones. In
\cite{MBY} the Donoho-Huber breakdown point of the $\ell_1$ estimator
was shown to be $0$, just as for the LSE. This result leaves the impression that 
the $\ell_1$-estimator is not robust at all, at least for random carriers.

The quantitative study of the robustness properties of the
$\ell_1$-estimator for non-random carriers (also called 
\emph{fixed design}), i.e., for a deterministic $X$, start with \cite{He-et-Al90}. 
In this work, the authors introduce a finite-sample measure of performance for
regression estimators based on tail behaviour. For the
$\ell_1$-estimator as well as for a class of M-estimators, 
their tail performance measure turns out to be equal to the RBP and they give a simple characterization of it in terms of the design
configuration. In particular, they show that the RBP of the $\ell_1$ estimator can be positive for non-random carriers.
The same expression for the RBP is obtained by Ellis and Morgenthaler \cite{Ellis-Mor}, who also study its role as a leverage measure.
Interestingly, these characterizations have been recently rediscovered in the context of the theory of compressed sensing, as we shall
see in Section \ref{S:link}.
From an optimization point of view, the same problem is studied by Giloni and Padberg \cite{Giloni-Padberg}, who provide a 
characterization of the RBP by using the concepts of $q$-{\em strength} and $s$-{\em stability} of a matrix, introduced by themselves. 
Additionally, they discuss uniqueness issues and their implications for the RBP.
These results reopen the discussion on the robustness of the $\ell_1$ estimator.

Problem \eqref{e:mainmodel} is reconsidered in \cite{Candes-Tao-Decoding} by signal processing specialists. 
Their work lies in the fixed design framework and they suppose, as in
\cite{He-et-Al90,Ellis-Mor,Giloni-Padberg}, that contamination is restricted to the dependent variable $y$. 
Moreover, they assume 
that the vector $\delta$ in \eqref{e:mainmodel} is \emph{sparse},  i.e., only a small fraction
of the observations is contaminated and the rest is completely free
of errors. This hypothesis, that would horrify any statistician,
permits to solve this problem via the successful theory of sparse
solutions to linear systems. This theory provides sufficient conditions for
\emph{exact} recovery of a signal from corrupted measurements.
The sufficient condition is known as the \emph{restricted isometry property} and it is verified with high probability for random 
normal matrices $X$  when $n$ and $p$ go to infinity in a proper ratio.

Later, in \cite{Candes-Randall}, a modification of $\ell_1$ minimization for linear regression is put forward in order to deal with outliers 
and noise. The sufficient conditions for the noiseless case are adapted to this more realistic context. 
However, their conditions are only sufficient and in the particular instance when $X$ is normal random and has orthonormal columns. 
A thorough study of $\ell_1$ minimization for struggling against noise coupled with outliers in linear regression is missing.

We perform a detailed error analysis of the $\ell_1$ estimator when the errors in \eqref{e:mainmodel} take the form  
$ \delta=z+e$, where $z$ is a noise term and $e$ is a sparse vector. As a consequence, we show that the RBP of the $\ell_1$ estimator 
characterizes the critical sparsity level of $e$ in order to exactly recovering $f$ in \eqref{e:mainmodel} by solving \eqref{P:l1} when $z=0$. 
The general conclusion of this analysis is that $\ell_1$ minimization manages remarkably well the presence of sparse outliers, but
has a poor response to noise. 

We introduce a new robust estimator that inherits the good properties of $\ell_1$ estimation and LSE for dealing simultaneously with 
outliers and noise, for a general matrix $X$. Our estimator is defined by a minimization problem involving the \emph{inf-convolution}
of the $\ell_1$ and $\ell_2$ norms of the residuals. A globally convergent algorithm for computing our estimator is proposed. A fine error analysis 
and numerical experiments corroborates that our estimator actually have a better behavior than LSE and $\ell_1$ estimator in face to noise and 
outliers. Moreover, in the absence of outliers or noise, our estimator reduces to LSE or $\ell_1$ estimator, respectively.

This paper is organized as follows.
In Section~\ref{S:link} we recall the contributions from the theory of sparse recovery to robust linear regression. In Section \ref{S:l1} we 
expound a detailed quantitave error analysis of the $\ell_1$ estimator. 
In Section~\ref{S:main} we introduce a new estimation technique that improves the error estimates of the $\ell_1$-estimator for data with noise. 
 In Section~\ref{S:algorithm} we present a globally convergent algorithm for
computing our estimator whose main advantage is its simplicity, as it is defined by a simple fixed-point iterative process. 
In Section~\ref{S:numerical} we provide numerical results confirming that our estimators
inherits the best of the LSE and $\ell_1$ estimator. We conclude the article with a summary and a discussion, presented in
Section \ref{S:conclusions}.

\subsection{Notation and preliminaries}\label{SS:preliminaries}
We shall use the notation $N=\{1,...,n\}$ for the index set of all the observations. For a set of indexes $M$,
$|M|$ denotes its cardinality. For a vector $x\in\RR^n$, we denote by
$\supp(x)$ its support, \emph{i.e.}, 
the index set of nonzero components, $\supp(x)=\{i\in N\mid x_i\neq 0\}$. The cardinality of the support 
of a vector, often called the ``$\ell_0$-norm'' or ``cardinality norm'', is denoted as $\|x\|_0$; thus 
$$\|x\|_0=|\{i\in N\mid x_i\neq 0\}|.$$ 
For a subset $M$ of $N$ and $p\in\left[1,\pinf\right[$, we define $\|\cdot\|_{p,M}\colon x\mapsto(\sum_{i\in M}|x_i|^p)^{1/p}$ and 
$\|\cdot\|_{\infty,M}\colon x\mapsto\max_{i\in M}|x_i|$. Moreover, for every $x\in\RR^n$ and $p\in\left[1,\pinf\right[$, we denote 
$\|x\|_p=\|x\|_{p,N}$ and $\|x\|_{\infty}=\|x\|_{\infty,N}$.

Let $\phi\colon\RR^n\to\RX$ be a lower semicontinuous convex function
which is proper in the sense that $\dom \phi=\menge{x\in\RR^n}{\phi(x)<\pinf}\neq\emp$.
The subdifferential operator of $\phi$ is
\begin{equation*}
\partial\phi\colon\RR^n\to 2^{\RR^n}\colon
x\mapsto\menge{u\in\RR^n}{(\forall
y\in\RR^n)\:u^\top (y-x)+\phi(x)\leq\phi(y)}
\end{equation*}
and we have  \cite[Theorem 2.2.1]{HU-Lemarechal93(I)}
\begin{equation}\label{E:Fermat}
x\in\Argmin{y\in\RR^n}{\phi(x)}\quad \Leftrightarrow\quad
0\in\partial\phi(x). 
\end{equation}
The proximal mapping associated with $\phi$ is defined by
\begin{equation}\label{E:def-prox}
\prox_{\phi}\colon\RR^n\to\RR^n\colon x\mapsto \argmin_{u\in\RR^n}
\left(
\phi(u) +\frac12\|u-x\|_2^2 \right).
\end{equation}
From \eqref{E:Fermat} we obtain
\begin{equation*}
p=\prox_{\phi}x\quad\Leftrightarrow\quad x-p\in\partial\phi(p),
\end{equation*}
and, since $\phi+\|\cdot-x\|^2/2$ is strongly convex, $\prox_{\phi}(x)$ exists and is unique for all
$x\in\RR^n$.

The following lemma will be useful throughout the paper.
\begin{lemma}\label{L:prox}
Let $\gamma\in\RPP$ and let
$\phi\colon\RR^n\to\RR\colon x\mapsto \phi(x)=\gamma\|x\|_1=\gamma\cdot \sum_{i=1}^n|x_i|$. 
Then the following hold.
\begin{enumerate}
\item\label{L:proxi}
For every $x\in\RR^n$,
\begin{equation*}
\partial\phi(x)=\overset{n}{\underset{i=1}{\times}}
\partial\gamma|\cdot|(x_i),
\end{equation*}
where
\begin{equation*}
(\forall\xi\in\RR)\quad \partial\gamma|\cdot|(\xi)=
\begin{cases}
\gamma,\quad&\text{if }\:\xi>0;\\
[-\gamma,\gamma],&\text{if }\:\xi=0;\\
-\gamma,&\text{if }\:\xi<0.
\end{cases}
\end{equation*}
\item\label{L:proxii} 
For every $x\in\RR^n$,
 $$\prox_{\gamma\phi}x=(\prox_{\gamma|\cdot|}(x_i))_{1\leq i\leq n},$$ where
\begin{equation*}
(\forall\xi\in\RR)\quad 
\prox_{\gamma|\cdot|}(\xi)=
\begin{cases}
\xi-\gamma&\text{if }\:\xi>\gamma;\\
0,&\text{if }\:\xi\in[-\gamma,\gamma];\\
\xi+\gamma,&\text{if }\:\xi<-\gamma.
\end{cases} 
\end{equation*}
\end{enumerate}
\end{lemma}
\begin{proof}
The results follow from \cite[Lemma~2.1, Lemma~2.9, and
Example~2.16]{Smms05}.\\
\end{proof}
\vskip .5\baselineskip

Recall the unique orthogonal decomposition of $z\in\RR^n$ as 
\begin{equation*}
z =X\bar{g}+\bar{b},\quad \bar{b}\in\ker(X^\top). 
\end{equation*}
When $X$ has full rank the \emph{hat matrix} 
\begin{equation*}
H=X(X^\top X)^{-1}X^\top 
\end{equation*}
is well defined and it holds that 
$$
\bar{b}=(I-H)z,\quad \bar{g}=(X^\top X)^{-1}X^\top z,\  \mbox{and}\  X\bar{g}=Hz.
$$

\section{Connections with sparse reconstruction and compressed sensing}\label{S:link}
In \cite{Candes-Tao-Decoding}, the problem of recovering an input $f$ from corrupted measurements 
\begin{equation}\label{E:linmod-sparse}
y=Xf+e,
\end{equation}
when the  error term $e$ is sparse, is considered. The goal was to solve this problem by exploiting 
recent advances in the study of Sparse Reconstruction Problems (SRP), which
consist in finding the sparsest solution to  underdetermined  linear systems. 

If we consider a matrix $F$ such that $\ker(F)=\ran(X)$, then from \eqref{E:linmod-sparse} we obtain $Fy=Fe$.
Let us denote $\tilde y=Fy$, and consider the following SRP
\begin{equation}\label{P:srp}
\begin{array}{lc}
\underset{s\in\RR^n}{\min} & \|s\|_0 \\ s.t & F s=\tilde{y}.
\end{array} 
\end{equation}
Clearly, $e$ is a feasible point for Problem \eqref{P:srp}. 
If it was the unique solution then it would be possible to recover the signal $f$  from $e$ by solving the system
\begin{equation*}
Xf=y-e.
\end{equation*}
This is indeed the case, under very mild assumptions on the sparsity of $e$, as the following Lemma shows.

\begin{lemma}
\label{l:2.1}
Let $\kappa(X)=\max \{|M|: \exists \theta \in \RR^p,\  X\theta \neq 0,\
\mbox{s.t}\ x_i^\top \theta =0\ \forall i\in M\subset\{1,\ldots,n\}
\}$, let $F$ be such that $\ker(F)=\ran(X)$, and set $\tilde{y}=Fy$.\\
 If  $$\|e\|_0\leq (n-\kappa(X)-1)/2,$$ then \eqref{P:srp} has unique solution $e$. 
\end{lemma}
\begin{proof}
Suppose that there is a vector $d\neq e$ with less than $(n-\kappa(X)-1)/2$ nonzero components such that $Fd=\tilde{y}=Fe$.
This implies that $F(d-e)=0$, i.e., $d-e\in\ker(F)=\ran(X)$. Hence, there exists $\theta\in\RR^p$ such that $d-e=X\theta\neq 0$. 
Since $\|d\|_0\leq (n-\kappa(X)-1)/2$ and $\|e\|_0\leq (n-\kappa(X)-1)/2$, $\|d-e\|_0\leq n-\kappa(X)-1$. 
Therefore, $d-e=X\theta$ has at least $\kappa(X)+1$ null components, which is in contradiction
with the maximality of $\kappa(X)$. 
\end{proof}

Under the conditions provided in Lemma~\ref{l:2.1} the problem of
recovering a signal from very incomplete information can be solved
via an SRP.
 Unfortunately, the problem of finding the sparsest
solution to linear systems is NP-hard (see, e.g., \cite{Candes-Tao-Decoding}). 
Therefore, a common approach consists in replacing  the $\ell_0$-norm by the
$\ell_1$-norm, which results in a convex (linear) optimization problem
that can be efficiently solved. The problem of determining if
this relaxation gives the sparsest solutions have been studied in
\cite{CRT,Candes-Tao-Decoding,Candes-Tao-Universal,Donoho04}
with positive results. In these works, the authors provide sufficient
conditions in order to obtain the sparsest solution via
$\ell_1$ minimization.
In \cite{Candes-Tao-Decoding}, Candes and Tao prove that $e$ actually 
is the unique solution to the convex problem 
\begin{equation}\label{P:sparse-decode}
\begin{array}{cc}
 \min\limits_{s\in\RR^n} & \|s\|_{1} \\ &  F(s-y)=0,
\end{array}
\end{equation}
provided that $F$ satisfies the following \emph{restricted isometry property} (RIP)
\begin{equation*}
(\exists q\in\{1,\ldots,n\})\quad
\delta_q+\theta_{q,q}+\theta_{q,2q}<1,
\end{equation*}
where
$$
\delta_q=\max_{|J|\leq q,c\in \RR^{|J|}} \left| \frac{\|F_J
c\|^2}{\|c\|^2} -1\right|
\quad\text{and}\quad \theta_{q,q'}=\max_{\substack{|J|\leq q,c\in
\RR^{|J|}\\|J'|\leq q',c'\in \RR^{|J'|}}} \frac{|\pto{F_J
c}{F_{J'}c'}|}{\|c\|\|c'\|}
$$ 
are the \emph{restricted isometry constants} of $F$ and $F_J$ denotes
the sub matrix of the columns of $F$ indexed by $J$. Let us further define 
\begin{equation*}
m_R(F)=\max\{q\in N\mid \delta_q+\theta_{q,q}+\theta_{q,2q}<1 \}. 
\end{equation*}
The following result gives a relation between the solution to
\eqref{P:sparse-decode} under condition RIP and the reconstruction of
$f$. 

\begin{theorem}[Candes and Tao \cite{Candes-Tao-Decoding}, Theorem
1.4]\label{Th:Candes-Tao}
Let  $y=Xf+e$ where $f\in \RR^p$ and $e\in\RR^n$,  and let $F$ be a matrix such that $FX=0$.
If $\|e\|_0\leq m_R(F)$, then $f$ is the unique solution to the
problem 
\begin{equation}\label{e:Cantao}
\min_{g\in \RR^p} \|y-Xg\|_{1}.
\end{equation}
\end{theorem}

Since then, Theorem \ref{Th:Candes-Tao} has been the object of several improvements. In \cite[Theorem 1]{Juditsky-Nemirovski} it is shown that 
$e$ is the unique solution to \eqref{P:sparse-decode} for any $\|e\|_0\leq k$ if and only if $\hat\gamma_k(F)<1/2$, where
$\hat\gamma_k(F)$ is defined as 
\begin{equation}\label{E:gammak}
\hat\gamma_k(F)=\max\limits_{s\in\RR^n}\max\limits_{\substack{M\subset N \\ |M|= k}} \left\{\sumover{M} |s_i|: \|s\|_1\leq 1, Fs=0 \right\}. 
\end{equation}
This result extends Theorem \ref{Th:Candes-Tao} by giving necessary and sufficient conditions for any given, deterministic
matrix.

However, model \eqref{E:linmod-sparse} is too simple. In practice one expects that all observations carry some noise.
 A more realistic model is
\begin{equation}\label{E:linmod-S2}
 y=Xf+z+e,
\end{equation}
where $z$ is a dense, presumably small, vector of noise and  $e$ is an arbitrary sparse vector.
Under this model, exact recovery is not longer possible. The goodness of an estimator is measured by 
its distance to some reference point, which can be the true parameter $f$, or some estimator of it.  
If there is a bound on that distance which is finite for any $e$ such that $\|e\|_0\leq k$, then the RBP 
of the estimator is at least $k$. 

In \cite{Candes-Randall}, the estimation problem is studied for the error model \eqref{E:linmod-S2}, in the particular
case of a matrix $X$ with orthonormal columns. They prove that the vector $f$ can be estimated from noisy measurements up to an additive factor 
by solving the convex problems (for $r=2$ or $r=\infty$)
\begin{equation*}
\begin{array}{cc}
\minimize{(g,b)\in \RR^p\times\RR^n}{\|y-Xg-b\|_1}\\
s.t \quad \|b\|_r\leq \sigma\\
\quad \quad X^\top b=0.
\end{array}
\end{equation*}
provided that $\|\bar{b}\|_r\leq \sigma$ and that additional conditions on the restricted isometry constants of $\sqrt{{n}/{p}}\ X^\top$ hold.

Nevertheless, a RIP-based analysis of this problem results unsatisfactory. It provides only sufficient conditions, which are very conservative 
in practice. Moreover, the only known matrices with a high value of $m_R$ are random matrices from normal or Bernoulli distributions. 
Also, it is not stable under linear transformations. For any given matrix $X$ one can find an invertible matrix $G$ in such a way 
that the RIP constants of $GX$ are arbitrarily bad, independently of those of $X$ \cite{Zhang}. This point  is particularly serious since 
a closer look at  \cite{Candes-Randall} shows that if the matrix $X$ does not have orthonormal columns, as is the case in 
statistical applications, the analysis would rely on the RIP constants of the matrix $GX^\top$, for $G=(X^\top X)^{-1}$.

In Section \ref{S:l1} we obtain sharp bounds on the estimation error of the $\ell_1$ estimator when the errors follow model \eqref{E:linmod-S2}
for a general matrix $X$. Our treatment is simple, transparent, and covers the cases with and without noise in a unified way. It serves as 
the basis for the improvement of the $\ell_1$ estimator presented in Section  \ref{S:main}.

\section{Characterization of the behaviour of the $\ell_1$-estimator faced to sparse outliers and noise}
\label{S:l1}

In this section we aim to study the problem of estimating, by $\ell_1$ minimization, the vector 
$f$ from observations of the form \eqref{E:linmod-S2}.
In our case, the matrix $X$ is only assumed to be 
of full rank and we provide deterministic and non-asymptotic error bounds for 
the estimator of $f$. In order to achieve these goals let us introduce some 
definitions and useful properties.

For a $n\times p$ matrix $X$, define for every $k\in\{1,\ldots,n\}$ the \emph{leverage constants} $c_k$ of $X$ as
\begin{equation}\label{E:def-ck}
c_k(X)= \min\limits_{\substack{M\subset N \\ |M|= k}} 
\min\limits_{\substack{ g\in \RR^p \\ g\neq 0}}
\frac{\sumover{N\setminus M} |x_i^{\top} g|}{\sumover{N}|x_i^{\top}
g|}=\min\limits_{\substack{M\subset N \\ |M|= k}} 
\min\limits_{\substack{ g\in \RR^p \\ \|g\|_2=1}}
\frac{\sumover{N\setminus M} |x_i^{\top} g|}{\sumover{N}|x_i^{\top}
g|}
\end{equation}
and
\begin{equation}\label{E:def-m(A)}
m(X)= \max\left\{k\in N\: \mid c_k(X)>\frac{1}{2} \right\}.
\end{equation}
Note that the two minima in \eqref{E:def-ck} are achieved since the
feasible set in both cases is compact and the objective function is
continuous. When there is not place for confusion, we shall omit the dependency of the constants $c_k$ on $X$.
\begin{lemma}\label{L:ck}
We have $c_0=1$, $c_n=0$ and, for every $k\in\{1,\ldots,n\}$, $c_k\leq
c_{k-1}$.
\end{lemma}
\begin{proof}
It is clear that $c_0=1$ and that $c_n=0$. Let $k\in\{1,\ldots,n\}$,
let $g\in\RR^p\setminus\{0\}$, and let $M$ with $|M|=k-1$ such
that
\[c_{k-1}=\frac{\sumover{N\setminus M} |x_i^{\top}
g|}{\sumover{N}|x_i^{\top}
g|}.\] 
Now let $i_0\in N\setminus M$ and let $\widetilde{M}=M\cup\{i_0\}$.
We have $|\widetilde{M}|=k$ and, from \eqref{E:def-ck} we obtain 
\begin{equation*}
c_{k-1} =\frac{\sumover{N\setminus M} |x_i^{\top}
g|}{\sumover{N}|x_i^{\top}g|}
        =\frac{\sumover{N\setminus\widetilde{M}} |x_i^{\top}
g|+|x_{i_0} g|}{\sumover{N}|x_i^{\top}g|} 
     \geq \frac{\sumover{N\setminus\widetilde{M}} |x_i^{\top}
g|}{\sumover{N}|x_i^{\top} g|}\geq c_k,
\end{equation*}
which yields the result.
\end{proof}
\begin{remark}\label{R:quantities}
Let $F$ be such that $\ker (F)=\ran (X)$ as in Section \ref{S:link}. Let $\hat\gamma_k(F)$ be defined in \eqref{E:gammak} and
$s_*(F)=\max\left\{k\in N\: \mid \hat\gamma_k(F)<\frac{1}{2} \right\}$.
These constants are related to $c_k(X)$ and $m(X)$ via $c_k(X)=1-\hat\gamma_k(F)$ and $m(X)=s_*(F)$.
\end{remark}
\vskip .5\baselineskip

Many of the results in this article rely on the following fundamental $\ell_1$ error estimate, which is 
an extension and refinement of \cite[Lemma 5.2]{He-et-Al90}.

\begin{lemma}\label{L:desigualdad-L1}
Let $X$ be a $n\times p$ real matrix, let $(c_k)_{1\leq k\leq n}$ and
$m(X)$ be defined as in \eqref{E:def-ck} and \eqref{E:def-m(A)},
respectively. In addition, let $M\subset N$, and let $y,b^*\in\RR^n$
and $g^*,g\in \RR^p$ be arbitrary. Then the following hold.
\begin{enumerate}
\item\label{L:L1i}  Suppose that $|M|=k<m(X)$. Then, 
\begin{equation*}
 \|y-X g-b^*\|_1-  \|y-Xg^*-b^* \|_1 \geq 
(2c_k-1) \|X(g-g^*)\|_1 -2 \sumover{N \setminus M} |y_i-x_i^{\top} g^*-b^*_i|.
\end{equation*}
\item\label{L:L1ii} Suppose that $|M|=0$. Then we have, for every
$b\in \RR^n$,
\begin{equation*}
 \|y-X g-b\|_1-  \|y-Xg^*-b^* \|_1 \geq \|
X(g-g^*)+b-b^*\|_1 -2 \sumover{N} |y_i-b^*_i- x_i^{\top} g^*|.
\end{equation*}
\end{enumerate}

\end{lemma}
\begin{proof}
\ref{L:L1i}: Let $y,b^*\in\RR^n$ and $g^*,g\in \RR^p$. We have, 
\begin{align*}
 \|y-Xg-b^*\|_1 &= \sumover{N} |y_i-x_i^{\top}g-b_i^*| \\
&= \sumover{N} |(y_i-x_i^{\top}g^*-b^*_i)-(x_i^{\top}g-x_i^{\top}g^*)|
\\
&= \sumover{N\setminus M}
|(x_i^{\top}g-x_i^{\top}g^*)-(y_i-x_i^{\top}g^*-b^*_i)| \\ 
&+\sumover{M} |(y_i-x_i^{\top}g^*-b^*_i)-(x_i^{\top}g-x_i^{\top}g^*)|
\end{align*}
and using the reverse triangle inequality $|u-v|\geq ||u|-|v||\geq |u|-|v|$ we obtain
\begin{align}\label{E:des1}
 \|y-Xg-b^*\|_1 &\geq 2\sumover{N\setminus
M}|x_i^{\top}g-x_i^{\top}g^*|-\sumover{N}|x_i^{\top}g-x_i^{\top}g^*|
\\ \nonumber 
&\quad+\sumover{N} |y_i-x_i^{\top}g^*-b^*_i| -2 \sumover{N\setminus
M}
|y_i-x_i^{\top}g^*-b^*_i|.
\end{align}

It follows from \eqref{E:def-m(A)} and \eqref{E:def-ck} that 
$c_k>1/2$ and there exists $g_k\neq g^*$ such that
\[ (\forall g,g^*\in\RR^p)\quad\text{s.t.}\quad g\neq g^*\quad
\frac{\sumover{N\setminus
M}|x_i^{\top}(g-g^*)|}{\sumover{N}|x_i^{\top}(g-g^*)|}\geq
\frac{\sumover{N\setminus
M}|x_i^{\top}(g_k-g^*)|}{\sumover{N}|x_i^{\top}(g_k
-g^*)|} =c_k,
\]
Thus,
$$
 \sumover{N\setminus M}|x_i^{\top}(g-g^*)|\geq c_k
\sumover{N}|x_i^{\top}(g-g^*)|.
$$
By replacing in \eqref{E:des1} we obtain:
\begin{equation*}
 \|y-Xg-b^*\|_1 - \|y-Xg^*-b^*\|_1 \geq (2c_k-1) \|X(g-g^*)\|_1
  -2 \sumover{N\setminus M} |y_i-x_i^{\top}g^*-b^*_i|
\end{equation*}
and the result holds.

\ref{L:L1ii}: The result is a direct consequence of the triangle
inequality of the $L1$ norm.
\end{proof}

Next, we provide an estimate for the reconstruction error of a solution $f_1$ to the $\ell_1$
minimization problem \eqref{e:Cantao} depending on the level of
contamination, including both noise and outliers. Since the least
squares estimator is optimal in the absence of outliers,
we measure the reconstruction error by comparing $f_1$ with $f_n$,
which is the least squares estimator in this case. More precisely, 
if $y_n:=y-e=Xf+z$ is the noisy part of the data, without outliers, and 
$z=X\overline{g}+\overline{b}$, with $\overline{b}\in Ker X^\top$ is the orthogonal decomposition of the
noise, the LSE on the data $y_n$ is $f_n=(X^\top X)^{-1}X^\top y_n=f+\overline{g}$. \\

\begin{theorem}\label{Th:estimate-pure-l1}
Let $y=Xf+z+e$ and $M=\supp(e)$. Suppose that 
$|M|=k\leq m(X)$. Consider the unique
decomposition of $z$ as $z=X\overline{g}+\overline{b}$, where
 $\overline{g}\in\RR^p$ and $\overline{b}\in Ker X^\top$, and let $f_n=f+\bar{g}$ as discused above. Then
the following hold for $f_1$.
\begin{enumerate}
\item 
If $\|\overline{b}\|_{\infty,N\setminus M}=0$, then $f_1=f_n$.
\item If $\|\overline{b}\|_{\infty,N\setminus M}>0$, then
\begin{equation}\label{E:l1-estimate}
\|X(f_1 -f_n)\|_1 \leq \frac{1}{2c_k-1} \left( \|\bar{b}\|_{1,N\setminus M} + 
\frac{\|\bar{b}\|_{2,N\setminus M}^2}{\|\bar{b}\|_{\infty,N\setminus M}}\right).
\end{equation}
\end{enumerate}
\end{theorem}

\begin{proof}
Using Lemma \ref{L:desigualdad-L1}\ref{L:L1i} with $b^*=0, g=f_1$, and
$g^*=f_n$ we obtain
  $$
 \|y-X f_1\|_1-  \|y-X f_n\|_1 \geq (2c_k-1)
\|X(f_1-f_n)\|_1 -2 \sumover{N \setminus M} |y_i- x_i^\top
f_n|. $$
Since, by hypothesis, $y_i=x_i^\top(f+\bar{g})+\bar{b}_i=x_i^\top f_n+\bar{b}_i$ for
$i\in N\setminus M$ we have
\begin{equation}\label{E:first-estimate}
(2c_k-1) \|X(f_1-f_n)\|_1 \leq 2 \|\bar{b}\|_{1,N \setminus M} + \|y-X f_1\|_1-  \|y-X f_n\|_1.
\end{equation}
 First note that, since $f_1$ is a minimizer, $\|y-X f_1\|_1-  \|y-X f_n\|_1\leq 0$, thus if $\|\bar{b}\|_{\infty,N \setminus M}=0$ it follows 
from \eqref{E:first-estimate} and the hypothesis of full rank on $X$ that $f_1=f_n$. Now suppose that $\|\bar{b}\|_{\infty,N \setminus M} > 0$.
By LP duality \cite[p. 1031-1032]{Giloni-Padberg},
$$
\|y-Xf_1\|_1=\min_{g\in\RR^p} \|y-Xg\|_1 =\max_{d\in P^*} d^\top y,
$$
where $P^*=\menge{d\in\ker X^{\top}}{\|d\|_{\infty}\leq 1}$ .
Thus, $$\|y-X f_1\|_1-  \|y-X f_n\|_1= \max_{d\in P^*} d^\top (e+\bar{b}) -  \|e+\bar{b}\|_1.$$
Hence, by using Lemma \ref{L:tecnico}, we obtain
\begin{align*}
\|y-X f_1\|_1-  \|y-X f_n\|_1 &\leq \|e+\bar{b}\|_{1,M}+\frac{\|\bar{b}\|_{2,N\setminus M}^2}{\|\bar{b}\|_{\infty,N\setminus M}} -  
\|e+\bar{b}\|_1\\
& = -\|\bar{b}\|_{1,N\setminus M}+\frac{\|\bar{b}\|_{2,N\setminus M}^2}{\|\bar{b}\|_{\infty,N\setminus M}}
\end{align*}
which altogether with \eqref{E:first-estimate} yields \eqref{E:l1-estimate}. 
\end{proof}

\begin{remark}\label{R:mejor}
Note that, by H\"older inequality, 
$\|\bar{b}\|_{2,N\setminus M}^2 \leq \|\bar{b}\|_{1,N\setminus M} \|\bar{b}\|_{\infty,N\setminus M}$  
then 
\begin{equation*}
\|\bar{b}\|_{1,N\setminus M} + 
\frac{\|\bar{b}\|_{2,N\setminus M}^2}{\|\bar{b}\|_{\infty,N\setminus M}}\leq 2 \|\bar{b}\|_{1,N\setminus M}\leq 2 \|\bar{b}\|_1.
\end{equation*}
\end{remark}

In the particular case when only sparse errors are present ($z=0$), the following result is a 
characterization of the exact recovery property, which improves Theorem \ref{Th:Candes-Tao} 
(see also \cite[Theorem 1]{Juditsky-Nemirovski} and \cite[Proposition 2.3]{Zhang} for related results).

\begin{theorem}
\label{Th:sparse-case} 
Let $f\in \RR^p$,  $e\in\RR^n$, and set
$y=Xf+e$. Then $f$ is the unique solution of the problem 
$$
\min_{g\in \RR^d} \|y-Xg\|_{1}.
$$
for any $\|e\|_{0}\leq k$ if and only if $k\leq m(X)$.
\end{theorem}
\begin{proof}
First note that, in this case, $f_n=f$. If  $\|e\|_{0}\leq m(X)$, by
using Theorem~\ref{Th:estimate-pure-l1} with $z=0$, we obtain that
$X(f_1-f_n)=X(f_1-f)=0$ and, since $X$ has full rank, we conclude that $f_1=f$.
Now let us show that for $k=\|e\|_0> m(X)$ we can find an instance of the
problem for which $f$, whether is not a solution, or it is not the unique solution.
Let $f\in\RR^p$ be arbitrary. From the definiton of $c_k$, there
exists $g_k\in \RR^p$ such that $\|g_k\|_2=1$ and $M\subseteq N, |M|=k$ such that 
\begin{equation}\label{E:contrad} 
\sumover{N\setminus M} |x_i^\top g_k| \leq \sumover{M}|x_i^\top g_k|.
\end{equation}

Now define, for $\alpha>0$,
\begin{equation*}
\overline{e}_i=\alpha\cdot
\begin{cases}
x_i^\top  g_k, &\text{if } i\in M;\\
0, &\text{otherwise}
\end{cases}
\end{equation*}
and $\overline{y}=Xf+\overline{e}$. Then,
\begin{eqnarray*}
\|\overline{y}-Xf\|_1=\alpha \sumover{M}|x_i^\top g_k|\\
\|\overline{y}-X(f+\alpha g_k)\|_1= \alpha \sumover{N\setminus M} |x_i^\top  g_k|.
\end{eqnarray*}
Hence, it follows from \eqref{E:contrad} that
$\|\overline{y}-X(f+\alpha g_k)\|_1\leq \|\overline{y}-Xf\|_1$, then $f+\alpha g_k$ is a minimizer.
\end{proof}

The proof of Theorem \ref{Th:sparse-case} shows that if $k>m(X)$, the for any $\alpha>0$ we can find a vector $e$ such that 
$\|e\|_0=k$ and the $\ell_1$ estimator $f_1$ on the data $y=Xf+\alpha e$ satisfies $\|f_1-f\|_2=\alpha$.
Combined with Theorem \ref{Th:estimate-pure-l1} this shows that the RBP of the $\ell_1$ estimator equals $m(X)$, recovering results from
\cite{Giloni-Padberg,Mizera-Muller99}.
At the same time, it shows the close relation between the concepts of regression breakdown point and exact recovery of sparse signals.
The most important consequence of this relation is the aproximation of the RBP of a given matrix using SemiDefinite Programming (SDP). 
Indeed, in \cite{ElGhaoui,Juditsky-Nemirovski} we can find SDP bounds on $\hat\gamma_k(F)$ that, in view of Remark \ref{R:quantities}, give 
lower bounds on the quantities $c_k(X)$, and thus on $m(X)$, which characterizes the RBP of a given matrix.

In the next section we introduce a new estimator for the 
model including sparse errors and noise. We also verify that this new 
estimator has a better performance compared to the $\ell_1$ estimator 
when dense noise and sparse errors are present.

\section{A robust estimator against sparse outliers and noise}\label{S:main}
In this section, we derive a new technique for estimating $f$ from 
\begin{equation*}
y=Xf+z+e,
\end{equation*}
where $z$ is a noise and  $e$ is an arbitrary sparse vector. Our estimator keeps the robustness of the $\ell_1$ estimator while improving its response to
noisy observations. In particular, in the absence of outliers, it reduces to the LSE. 

Theorem~\ref{Th:sparse-case} proves the efficacy of the $\ell_1$ estimator when dealing with sparse errors.
In contrast, Theorem~\ref{Th:estimate-pure-l1} highlights the drawbacks of this estimator when facing noisy observations. 
Since the LSE is the optimal choice when facing gaussian noise, it is natural to aim at combining their main strengths. 
The previous discussion motivates the following definition.

\begin{definition}
\label{prob:main}
Let $\sigma>0$, let $y\in\RR^n$, and let $X$ be a $n\times p$ real matrix with full rank. 
The $\ell_1\square\ell_2$ estimator is defined as the first component $\hat{g}$ of a solution to 
\begin{equation}\label{P:main-definition}
\begin{array}{cc}
\minimize{(g,b,s)\in
\RR^p\times\RR^n\times\RR^n}{\sigma\|s\|_{1}+\frac{1}{2}\|b\|^2_{2}}\\
s.t \quad y=Xg+b+s,
\end{array}
\end{equation}
where $b$ and $s$ are optimization variables estimating $z$ and $e$, respectively, and
$\sigma$ is an estimate of the magnitude of the noise.
\end{definition}

\begin{remark}
Note that \eqref{P:main-definition} can be set in the form of \eqref{P:M-estimators} with 
$$
\rho\colon r\mapsto\rho(r)=\inf_{b\in\RR^n} \sigma\|r-b\|_{1}+\frac{1}{2}\|b\|^2_{2}=\frac{1}{2}\|\cdot\|^2_{2}\square\|\cdot\|_1(r),
$$
where $h_1\square h_2=\inf_{u} \{h_1(\cdot-u)+h_2(u)\}$ denotes the \emph{inf-convolution} of $h_1$ and $h_2$
\cite{Rockafellar,HU-Lemarechal93(I)}.
In other words, Definition \ref{prob:main} amounts to defining the $\ell_1\square\ell_2$ estimator as a minimizer
 of the inf-convolution of the $\ell_1$ norm and the squared $\ell_2$ norm of the residuals. 
That is the reason for using the notation $\ell_1\square\ell_2$ for our estimator.
\end{remark}

Problem \eqref{P:main-definition} can be reduced by isolating $b$ or
$s$ from the linear constraint. 
This brings up the following two equivalent problems:
\begin{equation}\label{P:estimador-l1-1}
\minimize{(g,b)\in\RR^p\times\RR^n}{
\psi(g,b):=\sigma\|y-Xg-b\|_{1}+\frac{1}{2}\|b\|^2_{2}}.
\end{equation}
and
\begin{equation}\label{P:estimador-l1-2}
\minimize{(g,s)\in\RR^p\times\RR^n}{
\phi(g,s):=\sigma\|s\|_{1}+\frac{1}{2}\|y-Xg-s\|^2_{2}}
\end{equation}

Problems \eqref{P:estimador-l1-1} and \eqref{P:estimador-l1-2} are equivalent to  Problem \eqref{P:main-definition}.
The existence of solutions is ensured by the full rank condition on $X$, and the coercivity and continuity of the objective functions.
\vskip .5\baselineskip

Problem \eqref{P:estimador-l1-1} is more advantageous for analysis of theoretical properties of solutions, while
Problem \eqref{P:estimador-l1-2} is better adapted to be numerically solved. For these reasons we state and proof here 
the optimality conditions of Problem  \eqref{P:estimador-l1-1} and postpone the analysis of Problem  \eqref{P:estimador-l1-2}   
for Section \ref{S:algorithm}.


\begin{lemma}\label{L:optimality-conditions}
The following hold.
\begin{enumerate}
\item\label{L:optimality-conditionsi} $(\hat g,\hat b)$ is a solution to \eqref{P:estimador-l1-1} if and only if
$X^{\top} \hat b=0$ and
\begin{equation}\label{E:b_at_optimum}
(\forall i\in\{1,\ldots,n\})\quad \hat{b}_i=
\begin{cases}
\sigma,\quad&\text{if }y_i-x_i^{\top}\hat{g}>\sigma;\\
y_i-x_i^{\top}\hat{g},&\text{if }y_i-x_i^{\top}\hat{g}\in\left[-\sigma,\sigma\right];\\
-\sigma,&\text{if }y_i-x_i^{\top}\hat{g}<-\sigma.
\end{cases}
\end{equation}
In particular $\|\hat{b}\|_{\infty}\leq\sigma$.
\item\label{L:optimality-conditionsii} The dual of \eqref{P:estimador-l1-1} is 
\begin{equation}\label{e:dual}
\gamma:=\max_{u\in\sigma P^*}u^{\top}y-\frac{1}{2}\|u\|_2^2,
\end{equation}
where 
$P^*=\menge{u\in\ker X^{\top}}{\|u\|_{\infty}\leq 1}$
and $\min_{(g,b)\in\RR^p\times\RR^n}\psi(g,b)=\gamma$.
\end{enumerate}
\end{lemma}
\begin{proof}
Note that $\psi(g,b)$ can be equivalently written as
\begin{equation}
\label{e:repmatrix}
\psi(g,b)=\sigma\|y-\left[\begin{smallmatrix}X & I_n \end{smallmatrix}\right]\left(\begin{smallmatrix}g\\ b \end{smallmatrix}\right)\|_1
+\frac{1}{2}\|\left[\begin{smallmatrix} 0_p & I_n \end{smallmatrix}\right]\left(\begin{smallmatrix}g\\ b \end{smallmatrix}\right)\|_2^2
\end{equation}
where $I_n$ denotes the identity matrix of size $n\times n$ and $0_p$ the zero matrix of size $p\times p$.

\ref{L:optimality-conditionsi}:
Since \eqref{P:estimador-l1-1} is convex, a necessary and sufficient conditions for 
a solution $(\hat g,\hat b)$ to Problem \eqref{P:estimador-l1-1}
is
\begin{equation}\label{E:fermat}
0\in \partial \psi(\hat g,\hat b).
\end{equation}
Hence, by using \cite[Theorem 4.2.1]{HU-Lemarechal93(I)} in \eqref{e:repmatrix} (qualification conditions are trivially satisfied), 
\eqref{E:fermat} is equivalent to
\begin{equation*}
\left(\begin{smallmatrix}0\\ 0 \end{smallmatrix}\right)\in - \left[\begin{smallmatrix}X^T\\ I_n \end{smallmatrix}\right] \partial\sigma \|\cdot\|_1(y-X\hat g-\hat b)+\left(\begin{smallmatrix}0\\\hat b \end{smallmatrix}\right).
\end{equation*}
Therefore, there exists $u\in \partial \sigma \|\cdot\|_1(y-X\hat g-\hat b)$ such that
\begin{equation*}
\begin{cases}
X^Tu=0,\\
\hat b=u
\end{cases}
\end{equation*}
or, equivalently,
\begin{equation*}
\begin{cases}
\hat b \in \partial \sigma \|\cdot\|_1(y-X\hat g-\hat b),  \\
X^T \hat b=0. 
\end{cases}
\end{equation*}
Hence
\[y-X\hat{g}-\hat{b}=\prox_{\sigma \|\cdot\|_1}(y-X\hat{g}),\]
and the result follows from Lemma~\ref{L:prox}\ref{L:proxii}.  
\ref{L:optimality-conditionsii}: Problem \eqref{P:estimador-l1-1} is equivalent to \eqref{P:main-definition} and, applying Lagrangian duality, the 
dual is \begin{equation*}
\max_{u\in\RR^p}\min_{(g,b,s)\in\RR^d\times\RR^n\times\RR^n}\sigma\|s\|_1+\frac12\|b\|_2^2+u^{\top}(y-Xg-b-s)
\end{equation*}
or, equivalently,
\begin{equation}
\label{e:aux23}
\max_{u\in\RR^p}\Big(u^{\top}y+\big(\min_{b\in\RR^n}\frac12\|b\|_2^2-u^{\top}b\big)
+\big(\min_{s\in\RR^n}\sigma\|s\|_1-u^{\top}s\big)-\max_{g\in\RR^p}g^{\top}(X^{\top}u)\Big).
\end{equation}
The optimality conditions associated to the convex optimization problem
\begin{equation*}
\min_{b\in\RR^n}\frac12\|b\|_2^2-u^{\top}b
\end{equation*}
 yields $b=u$, hence $\min_{b\in\RR^n}\frac12\|b\|_2^2-u^{\top}b=-\frac12\|u\|_2^2$. The second minimization problem can be written as
 \begin{equation*}
 \min_{s\in\RR^n}\sigma\|s\|_1-u^{\top}s=\sum_{i=1}^n \min_{s_i\in\RR}\sigma|s_i|-u_is_i
 =\begin{cases}
 -\infty,&\text{if }\|u\|_{\infty}>\sigma;\\
 0,&\text{if }\|u\|_{\infty}\leq\sigma.
 \end{cases}
 \end{equation*} 
 Finally, we have
 \begin{equation*}
 \max_{g\in\RR^p}g^{\top}(X^{\top}u)=\begin{cases}
  \pinf,&\text{if }u\notin\ker X^{\top};\\
  0,&\text{if }u\in\ker X^{\top}.
  \end{cases}
 \end{equation*}
 Altogether, it follows from \eqref{e:aux23} that the dual to \eqref{P:estimador-l1-1} is given by \eqref{e:dual} and the absence of duality gap 
 follows from the Slater qualification condition and the existence of multipliers \cite[section~4]{HU-Lemarechal93(I)}. 
\end{proof}


Let us show that, in the absence of sparse errors ($e=0$), the solution to Problem~\ref{prob:main} actually coincides with the LSE.
\begin{proposition}
\label{prop:lsrec}
Let $y=Xf+z$ and consider the unique decomposition of $z$ as $z=X\overline{g}+\overline{b}$, where
$\overline{g}\in\RR^p$ and $\overline{b}\in Ker X^\top$. 
If $\|\overline{b}\|_{\infty}<\sigma$, then
$(f_{n},\overline{b})$ is the unique solution to \eqref{P:estimador-l1-1}.
 \end{proposition}
\begin{proof}
Let us first prove that $(f_n,\bar{b})$ is a solution.
By definition $X^\top \overline{b}=0$, therefore it is enough to prove that 
$$
 \overline{b} \in \partial \sigma \|\cdot\|_1(y-X f_{n}- \overline{b}).$$
 Since $y=X(f+\bar{g})+\bar{b}=Xf_n+\bar{b}$, then $y-X f_{n}-\overline{b}=0$
and  $\sigma \partial \|\cdot\|_1(y-Xf_{n}-\overline{b})= [-\sigma,\sigma]^n.$
By the hypothesis on $\overline{b}$  we conclude that  
$\overline{b} \in \partial\sigma \|\cdot\|_1(y-Xf_{n}-\overline{b})$. 
Now let us prove that $(f_n,\bar{b})$ is the unique solution. 
Let $\varphi:\RR^p\to\RR$ be the continuous function defined by $\varphi(g)=\|y-Xg\|_\infty$.
By hypothesis, $\varphi(f_n)=\|\bar{b}\|_\infty<\sigma$ and the continuity of $\varphi$ yields the existence of a neighbourhood $V$ of $f_n$ such that $\varphi(g)<\sigma$ for every $g\in V$. 
Now let $(g,b)$ be a pair in $V \times\RR^n$ satisfying \eqref{E:b_at_optimum}. Then, since \eqref{E:b_at_optimum} yields $y=Xg+b=Xf_n+\bar{b}$, it follows from \eqref{P:estimador-l1-1}  that
 $$
 \begin{aligned}
\psi(g,b)-\psi(f_n,\bar{b})&=\sigma(\|y-Xg-b\|_1-\|y-Xf_n-\bar{b}\|_1)+\frac12\|b\|_2^2-\frac12\|\bar{b}\|_2^2\\
&=\frac12 \|y-Xg\|_2^2-\frac{1}{2} \|\bar b\|_2^2 \\
&= \frac12 \|X(f_n-g)\|_2^2+\bar{b}^\top X (f_n-g)\\
&= \frac12 \|X(f_n-g)\|_2^2\geq0.
\end{aligned}
$$
Since this inequality is valid for any candidate to solution close enough to $(f_n,\bar{b})$, the 
uniqueness follows from the convexity of $\psi$ and the full rank of $X$.
\end{proof}

Proposition~\ref{prop:lsrec} provides an interpretation of $\sigma$ as a threshold of the significance of outliers. 
Indeed, the part of the residuals that is below $\sigma$ is considered as noise, and the rest as outlier. 
If most of the outliers are comparable to $\sigma$ in magnitude, they can be perceived as noise, and 
the $\ell_1\square\ell_2$ estimator is close to the LSE. Moreover, as $\sigma$ goes to  $0$, our estimator tends to $f_1$.


We pursue the study of our estimator by showing that the additional term $b$, which makes the difference between 
our estimator and the $\ell_1$ estimator, improves its error bounds.
The numerical simulations performed in Section~\ref{S:numerical} confirm that the additional term actually plays 
an important role reducing the bias induced by noise. 
\begin{theorem}\label{Th:main}
Let $y=Xf+z+e$, let $M=\supp(e)$, and suppose that $|M|=k\leq m(X)$. Consider the unique
decomposition of $z$ as $z=X\overline{g}+\overline{b}$, where
$\overline{g}\in\RR^p$ and $\overline{b}\in Ker X^\top$. Then any solution  $(\hat g,\hat b)$ to
\eqref{P:estimador-l1-1} satisfies 
\begin{equation*}
\|X(\hat{g} -f_n)\|_1 \leq \frac{1}{2c_k-1} \left( \|\bar{b}-\hat{b}\|_{1,N\setminus M} + 
\frac{\|\bar{b}-\hat{b}\|_{2,N\setminus M}^2}{\|\bar{b}-\hat{b}\|_{\infty,N\setminus M}}\right),
\end{equation*}
where $f_n=f+\bar{f}$ is the LSE on $y_n=Xf+z$. 
\end{theorem}

\begin{proof}
 From Lemma~\ref{L:desigualdad-L1}\ref{L:L1i} and \eqref{P:estimador-l1-1} we deduce
\begin{equation*}
\psi(\hat g,\hat b) -\psi(f_{{n}},\hat b)\geq \sigma (2c_k-1) \|X (\hat g-f_{{n}})\|_1-2\sigma 
\|y-Xf_{n}-\hat b\|_{1,N\setminus M}.
\end{equation*}
Hence, it follows from $f_n=f+\bar{g}$ that, for every $i\in\{1,\ldots,n\}$, $y_i-x_i^{\top}f_n=e_i+\bar{b}_i$ and, thus, 
$\psi(f_{{n}},\hat b)=\sigma\|e+\bar{b}-\hat{b}\|_1+\|\hat{b}\|_2^2/2$. Therefore, since $e_i=0$ for any $i\in N\setminus M$,  
\begin{equation}
\label{e:auxibound}
\sigma (2c_k-1) \|X (\hat g-f_{{n}})\|_1 \leq 2\sigma 
\|\bar{b}-\hat b\|_{1,N\setminus M}-\sigma \|e+\bar{b}-\hat b \|_1+\psi(\hat g,\hat b) 
-\frac12 \|\hat{b}\|_2^2.
\end{equation}
From Lemma~\ref{L:optimality-conditions}\ref{L:optimality-conditionsii}, the dual problem to \eqref{P:estimador-l1-1} is
\begin{equation*}
\max_{u\in \sigma P^*} u^\top (e+\bar{b})-\frac12 \|u\|_2^2
\end{equation*}
and $\psi(\hat g,\hat b)=\max\limits_{u\in \sigma P^*} u^\top (e+\bar{b})-\frac12 \|u\|_2^2$. 
Therefore
\begin{align*}
\psi(\hat g,\hat b) -\frac12 \|\hat{b}\|_2^2=&\max_{u\in \sigma P^*} u^\top (e+\bar{b})-\frac12 \|u\|_2^2 -\frac12 \|\hat{b}\|_2^2\\
=& \max_{u\in \sigma P^*} u^\top (e+\bar{b}-\hat{b}) -\frac12 \|u-\hat{b}\|_2^2\\
\leq & \max_{u\in \sigma P^*} u^\top (e+\bar{b}-\hat{b}).
\end{align*}
Hence, it follows from Lemma \ref{L:tecnico} that
$$
\psi(\hat g,\hat b) -\frac12 \|\hat{b}\|_2^2\leq  
\sigma\|e+\bar{b}-\hat{b}\|_{1,M}+
\frac{\sigma}{\|\bar{b}-\hat{b}\|_{\infty,N\setminus M}}\|\bar{b}-\hat{b}\|_{2,N\setminus M}^2,
$$
which, combined with \eqref{e:auxibound}, yields 
\begin{align*}
(2c_k-1) \|X (\hat g-f_{{n}})\|_1 &\leq 2 \|\bar{b}-\hat b\|_{1,N\setminus M}- \|e+\bar{b}-\hat b \|_1+
\|e+\bar{b}-\hat{b}\|_{1,M}\nonumber\\
&\hspace{3cm}+
\frac{1}{\|\bar{b}-\hat{b}\|_{\infty,N\setminus M}}\|\bar{b}-\hat{b}\|_{2,N\setminus M}^2\\
&= \|\bar{b}-\hat b\|_{1,N\setminus M}+
\frac{1}{\|\bar{b}-\hat{b}\|_{\infty,N\setminus M}}\|\bar{b}-\hat{b}\|_{2,N\setminus M}^2
\end{align*}
as claimed.
\end{proof}

Note that the bound in Theorem~\ref{Th:main}  depends only on the data of the problem and, in particular, it does not depend explicitly on $\sigma$. 
The only dependency is through $\hat{b}$, which is bounded by $\sigma$. 


The following result gives a connection between the RBP of the $\ell_1\square\ell_2$ estimator and that of the $\ell_1$ estimator.
We recall that the regression breakdown point of an estimator is the maximum number of components 
of the data $y$ that may diverge while keeping the estimator bounded.
\begin{corollary}
The RBP of the $\ell_1\square\ell_2$ estimator is at least $m(X)$. 
\end{corollary}
\begin{proof}
Let us prove that any minimizer $\hat g$ of the function 
 $\psi$ defined in \eqref{P:estimador-l1-1} is bounded, no matter how large the sparse term $e$ is.
 From Theorem~\ref{Th:main} and Lemma~\ref{L:optimality-conditions}\ref{L:optimality-conditionsi}, if $k=\|e\|_0\leq m(X)$ we obtain the estimate
 $$
 \|X(\hat g -f_{n})\|_1 \leq \frac{n\sigma+\|\bar{b}\|_1}{ c_k-1/2}<\pinf,
 $$
and the result follows.
\end{proof} 

\section{Algorithm}\label{S:algorithm}

In this section we propose and study an algorithm  for computing the estimator introduced in the previous section, 
which is an application of the forward-backward splitting method \cite{Opti04,Smms05}. 
Note that problem \eqref{P:estimador-l1-2} can be written equivalently as 
 \begin{equation}\label{P:estimador-l1-21}
\min_{s\in \RR^n}\Big(\sigma\|s\|_{1}+ \min_{g\in\RR^p} \frac{1}{2}\|y-Xg-s\|^2_{2}\Big).
\end{equation}
The first-order  optimality condition of the inner problem in $g$ yields $X^{\top}(Xg+s-y)=0$ or 
equivalently 
\begin{equation}
\label{e:calculog}
g=(X^ {\top} X)^{-1}X^ {\top} (y-s).
\end{equation}
Hence, from \eqref{P:estimador-l1-21} we obtain
\begin{equation}\label{P:estimador-l1-3}
 \min_{s\in \RR^n}\sigma\|s\|_{1}+\frac{1}{2}\|(I-H)(y-s)\|^2_{2},
\end{equation}
where $H=X(X^ {\top} X)^{-1}X^ {\top} $ is the hat matrix. Moreover, since the objective function in \eqref{P:estimador-l1-3} is the sum of a 
general convex function and a differentiable convex function with Lipschitz gradient,
the solutions of \eqref{P:estimador-l1-3} are characterized
\cite[Proposition 3.1]{Smms05} 
by
\begin{equation}\label{E:characterization_of_s}
(\forall \gamma>0)\quad
s=P_{\gamma\sigma}(s-\gamma(I-H)(y-s)), 
\end{equation}
where $P_{\gamma\sigma}:=\prox_{\gamma\sigma\|\cdot\|_1}$
is the  proximal operator asociated to the function 
$\gamma\sigma \|\cdot\|_1,$  defined 
in \eqref{E:def-prox}. By using Lemma \ref{L:prox} we obtain that for every $\gamma>0$, 
\begin{align}\label{E:prox-l1}
P_{\gamma}\colon\RR^n&\to\RR^n\nonumber\\
(\xi_1,\ldots,\xi_n)&\mapsto(\sign(\xi_1)\max\{\gamma-|\xi_1|,0\},
\ldots,\sign(\xi_n)\max\{\gamma-|\xi_n|,0\}).
\end{align}
Combining the fixed-point characterization \eqref{E:characterization_of_s} with the expression for the 
proximal mapping \eqref{E:prox-l1} and adding relaxation steps $(\lambda_k)_{k\in\NN}$, we
obtain Algorithm~\ref{A:fb}.
\begin{algorithm}[h!]
\begin{tabular}{l@{\extracolsep{.5cm}}p{12cm}}
& Choose $s_0\in \RR^n$ and set $k := 0$. Iterate: \\
\midrule
1.& Choose $0<\lambda_k \leq 1$ and $0< \gamma_k<2/\|I-H\|$.\\
2.& Let
\begin{equation*}
s_{k+1}=s_k+\lambda_k\big(P_{\sigma\gamma_k} (s_k-\gamma_k(I-H)(s_k-y))-s_k\big).
\end{equation*} \\
3. & If a stopping criterion is satisfied, stop. Otherwise set $k := k + 1$ and go to step 1.\\
\end{tabular}
\caption{The forward-backward algorithm for solving \eqref{P:estimador-l1-3}.}\label{A:fb}
\end{algorithm}

\noindent The convergence properties of Algorithm \ref{A:fb} are stated in the following Theorem.

\begin{theorem}
\label{t:FBmms2}
Let $(\gamma_k)_{k\in\NN}$ be a sequence in $\RPP$ such that 
$0<\inf_{k\in\NN}\gamma_k\leq\sup_{k\in\NN} \gamma_k<2/\|I-H\| $ and let
$(\lambda_k)_{k\in\NN}$ be a sequence in $]0,1]$ such that $\inf_{k\in\NN}\lambda_k>0$.
Let $s_0\in\RR^n$ be arbitrary, and let $(s_{k})_{k\in\NN}$ be the sequence of iterates 
obtained from Algorithm~\ref{A:fb}. Then the following hold.
\begin{enumerate}
 \item $(s_k)_{k\in\NN}$ converges to a solution $\hat{s}$ to \eqref{P:estimador-l1-3}.
 \item $\sum_{k\in\NN}\|(I-H)(s_k-\hat{s})\|_2^2<\pinf$.
 \end{enumerate}
\end{theorem}
\begin{proof}
Note that $\|\cdot\|_1$ is a convex continuous function and $\frac{1}{2}\|(I-H)(y-\cdot)\|^2_2$ is a convex
differentiable function with gradient
\begin{equation}
\label{e:aux2}
\nabla\Big(\frac{1}{2}\|(I-H)(y-\cdot)\|^2_2\Big) =(I-H)(\cdot-y), 
\end{equation}
where the last equality follows from the projector property of $H$, $(I-H)^*(I-H)=(I-H)^2=(I-H)$. 
We deduce from \eqref{e:aux2} that
$\nabla(\frac{1}{2}\|(I-H)(y-\cdot)\|^2_2)=(I-H)(\cdot-y)$ is $L$-Lipschitz continuous with $L=\|I-H\|$.
Moreover it follows from \eqref{E:prox-l1} that, for every $\gamma\in\RPP$, 
$\prox_{\gamma\|\cdot\|_1}=P_{\gamma}$.
Altogether the results follow from \cite[Theorem~3.4]{Smms05}.
\end{proof}

Theorem~\ref{t:FBmms2} asserts that Algorithm~\ref{A:fb} approximates a solution $\hat{s}$ to \eqref{P:estimador-l1-3}. 
Since \eqref{P:estimador-l1-3} is equivalent to \eqref{P:main-definition}, 
it follows from \eqref{e:calculog} that the $\ell_1\square\ell_2$ estimator can be computed as
\begin{equation*}
\hat{g}=(X^ {\top} X)^{-1}X^ {\top} (y-\hat{s}).
\end{equation*} 
\section{Numerical Experiments}\label{S:numerical}
  
As announced in Section \ref{S:main}, numerical experiments confirm that the new estimator have lower bias when compared
 to $\ell_1$ or LSE estimation. In this section we describe the experimental setup and present numerical results.
  
The matrix $X$ is generated randomly with independent entries drawn from a standard normal distribution.
Its size is $n\times p=512\times 128$. The vector of data is generated according to 
$$ y=Xf+z+e,$$ with $f=0$ and $z$ standard normal, for different types and levels of contamination.
 
We estimate $f$ by three different methods: LSE, $\ell_1$, and $\ell_1\square\ell_2$, with $\sigma=\sqrt{\chi^2_1(.95)}$. The size of the support of 
$e$ ranges from $1$ to $(n-p-1)/2$, which means that the maximum fraction of contamination is close to $40\%$. We consider three types of sparse 
contamination. In the first and second types, each non-zero component of $e$ is drawn i.i.d. from a Normal (light-tailed) and Laplace (heavy-tailed) 
distribution with mean $0$ and standard deviation $5$, respectively. The last type of sparse error is considered to be very large and adversarial.
For generating the adversarial contamination we first create the vector $\tilde{e}=X\mathbbm{1}_p$, where $\mathbbm{1}_p$ is the vector of ones of size 
$p\times 1$. Then the sparse errors are obtained by selecting some components of $\tilde{e}$ randomly and by multiplying them by $50$.

For each type of contamination, for every $k\in\{1,\ldots,(n-p-1)/2\}$, we repeat $1000$ times the following:
\begin{enumerate} 
\item[1)] Choose randomly a subset $M$ of $N$ of size $k$.
\item[2)] Construct the sparse vector $e$ by filling the entries indexed by $M$ with the corresponding type of large errors.
\item[3)] Generate $z$ with independent $N(0,1)$ entries.
\item[4)] Set $y=z+e$ and estimate $f=0$ by LSE, $\ell_1$, and $\ell_1\square\ell_2$ methods.
\end{enumerate}

\noindent For computing the $\ell_1\square\ell_2$ estimator the algorithm described in Section~\ref{S:algorithm} is used. The code is 
available at \texttt{http://www.dim.uchile.cl/$\sim$sflores}. The $\ell_1$ estimator is computed by solving  
an equivalent linear program using the GNU solver \texttt{glpk}.\\
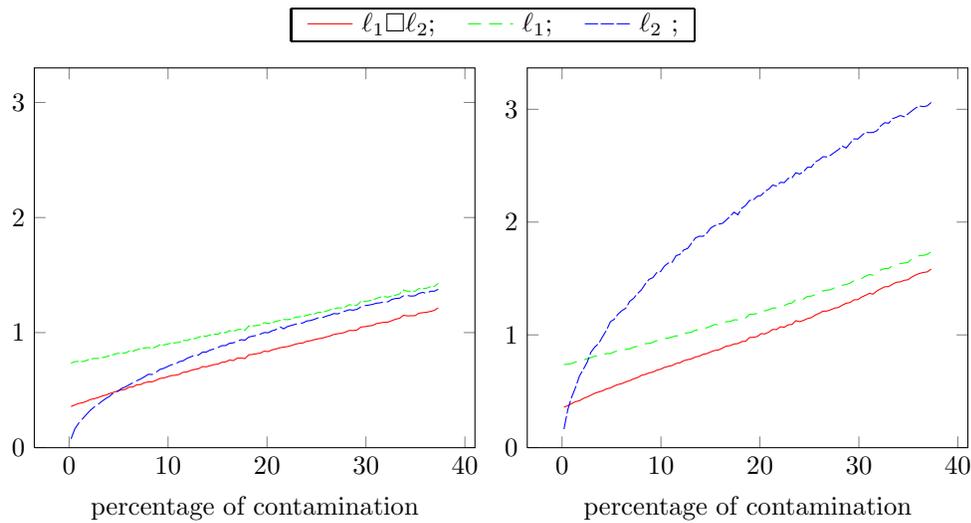
\begin{figure}[t]
\begin{center}
\begin{tabular}{|ccc|}
\hline
\ref{pgfplot:l1l2} $\ell_1\square\ell_2$;& \ref{pgfplot:l1} $\ell_1$;& \ref{pgfplot:l2} $\ell_2$ ;\\
\hline
\end{tabular}
\end{center}

\vspace*{.3cm}
\begin{tikzpicture}
\begin{axis}[
width=.45\textwidth,
scale only axis,
 xlabel=percentage of contamination,
 ymin=0]
\addplot[color=red,solid] file {dataP2.txt} ;
\addplot[color=green,dash pattern=on 3pt off 1pt] file {dataP2l1.txt};
\addplot[color=blue,dash pattern=on 7pt off 1.5pt] file {dataP2ls.txt} ;
\addplot[color=white,solid,thin] coordinates{(5,3)};
\end{axis}
\end{tikzpicture}
\begin{tikzpicture}
\begin{axis}[
width=.45\textwidth,
scale only axis,
    xlabel=percentage of contamination,
 ymin=0]
\addplot[color=red,solid] file {dataP1.txt} ; \label{pgfplot:l1l2} 
\addplot[color=green,dash pattern=on 5pt off 3pt] file {dataP1l1.txt};\label{pgfplot:l1}
\addplot[color=blue,dash pattern=on 5pt off 1pt] file {dataP1ls.txt} ;\label{pgfplot:l2}
\end{axis}
\end{tikzpicture}
 \caption{Relative error $\|\hat f-f_{{n}}\|/\|f_{{n}}\|$ for different percentage of outliers with gaussian noise. On the left, the contamination is drawn from a $N(0,5)$ distribution and on the right from a Laplace $(0,5)$ distribution.}\label{F:numerical1}
 \end{figure}
In Figure~\ref{F:numerical1} the bias for data with gaussian noise and sparse contamination is plotted. For each percentage of outliers 
the bias is quantified by the mean of the quotients  $\|\hat f-f_{{n}}\|_2/\|f_{{n}}\|_2$, where $\hat f$ is the estimation of $f$ 
obtained by each of the three methods and $f_{{n}}=(X^\top X)^{-1}X^\top z$.
In the figure on the left the bias is plotted for different levels of contamination with light-tailed outliers. We perceive that 
LSE outperforms $\ell_1\square\ell_2$ estimator when
the vector of outliers is very sparse (less than 5\% of contamination) and, hence, the gaussian noise predominates. However, the $\ell_1\square\ell_2$ estimator has a lower bias in general. Notice that the difference of the bias between LSE and $\ell_1$ estimator decreases as the percentage of contamination raises. In the figure on the right the bias is plotted  
for different levels of contamination with heavy-tailed outliers. In this case we observe the much better performance of the $\ell_1\square\ell_2$ estimator with respect to LSE even for very low levels of sparse contamination. Notice that, in this case, the $\ell_1$ estimator outperforms LSE for almost any percentage of contamination. In Figure~\ref{F:numerical2} we plot the bias under gaussian noise and very large adversarial sparse errors. On the left, we observe that the $\ell_1\square\ell_2$ estimator outperforms dramatically LSE for any level of contamination and, on the right, we focus on the low contamination zone for perceiving the better performance of the $\ell_1\square\ell_2$ estimator with respect to the $\ell_1$ estimator. In addition, we appreciate a clear breakdown phenomenon when the level of contamination exceeds the $30\%$ approximately.

In summary, we perceive the high sensitivity of LSE with respect to the percentage of outliers and, in special, with respect to heavy-tailed and 
adversarial ones. In every examined case we confirm the better performance of the $\ell_1\square\ell_2$ estimator with respect to the $\ell_1$ 
estimator, as expected in view of Theorem~\ref{Th:main}.

\begin{figure}[t]
\begin{center}
\begin{tabular}{|ccc|}
\hline
\ref{pgfplot2:l1l2} $\ell_1/\ell_2$;& \ref{pgfplot2:l1} $\ell_1$;& \ref{pgfplot2:l2} $\ell_2$ ;\\
\hline
\end{tabular}
\end{center}

\vspace*{.3cm}
\begin{tikzpicture}
\begin{axis}[
width=.45\textwidth,
scale only axis,
    xlabel=percentage of contamination,
 ymin=0]
\addplot[color=red,solid] file {dataP3.txt} ; \label{pgfplot2:l1l2} 
\addplot[color=green,dash pattern=on 5pt off 3pt] file {dataP3l1.txt};\label{pgfplot2:l1}
\addplot[color=blue,dash pattern=on 5pt off 1pt] file {dataP3ls.txt} ;\label{pgfplot2:l2}
\end{axis}
\end{tikzpicture}
\begin{tikzpicture}
\begin{axis}[
width=.45\textwidth,
scale only axis,
xlabel=percentage of contamination,
ymin=0]
\addplot[color=red,solid] file {dataP3-short.txt} ;
\addplot[color=green,dash pattern=on 3pt off 1pt] file {dataP3l1-short.txt};
\end{axis}
\end{tikzpicture}
 \caption{Plot of relative error $\|\hat f-f_{{n}}\|/\|f_{{n}}\|$, with $z$ standard gaussian, for different fractions of
 gross errors; at left, with adversarial contamination in the order of $50$; at right a closeup comparing $\ell_1\square\ell_2$
 and $\ell_1$ on the zone of low contamination.}\label{F:numerical2}
 \end{figure}
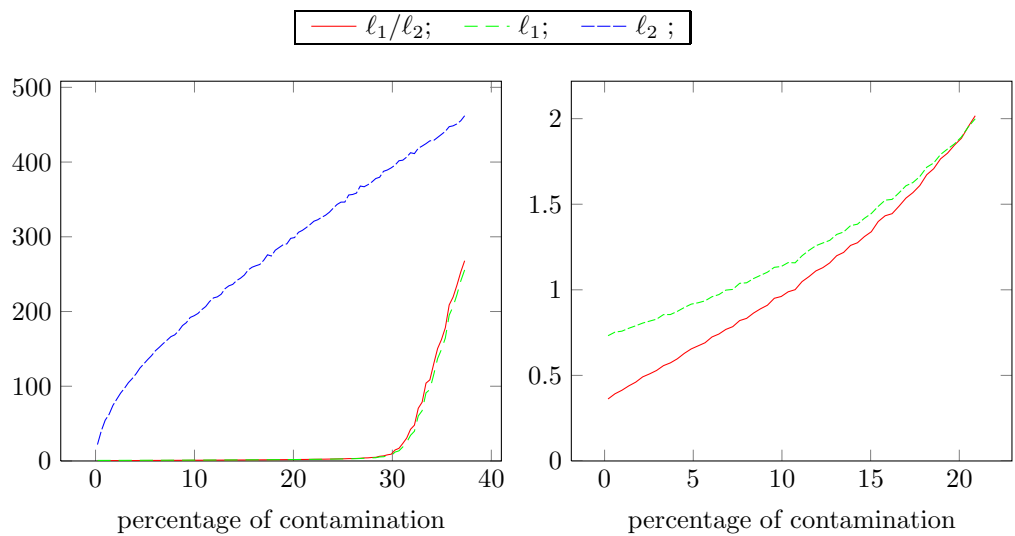

\section{Conclusions}\label{S:conclusions}

We have studied in deep the connections between robust regression and sparse reconstruction.
This link between apparently unrelated areas is of great interest for specialists as it permits 
to feed from each other of results, techniques as well as of new problems and questions.
The results presented in this article are quantitative, in contrast with the qualitative 
(bounded/unbounded) character of the results prevailing in robust statistics. 
We provide necessary and sufficient conditions, deterministic and for general data, in contrast 
with previous works based on restricted isometries. Our approach is simple and transparent, 
but powerful enough to treat the noisy case without modifications.

We have introduced a new estimator that combines the robustness of the $\ell_1$ estimator with 
the nice properties of the LSE. Numerical experiments show that the $\ell_1\square\ell_2$ estimator
behaves like the $\ell_1$ estimator concerning robustness to large outliers, but is less influenced 
by noise.

\section*{Appendix}

\begin{lemma}
\label{L:tecnico}
Let $b\in \RR^n$,  $e\in\RR^n$ and let $M=\supp(e)$. Suppose that $|M|\leq m(X)$ and $\max\limits_{i\in N\setminus M}|b_i|>0$. Let us define
\begin{equation}
\label{e:pstar2}
P^*=\menge{d\in\ker X^{\top}}{\|d\|_{\infty}\leq 1}.
\end{equation}
Then, for every $\sigma>0$,
\begin{equation*}
\max_{d\in\sigma P^*}d^{\top}(e+b)\leq\sigma\|e+b\|_{1,M}+\frac{\sigma}{\|b\|_{\infty,N\setminus M}}\|b\|_{2,N\setminus M}^2.
\end{equation*}
\end{lemma}
\begin{proof}
Let
\begin{equation}
\label{e:deftilda}
\tilde{b}_i=
\begin{cases}
0, &\text{if } i\in M;\\
b_i, &\text{otherwise},
\end{cases}
\quad \mbox{and} \quad 
\tilde{e}_i=
\begin{cases}
b_i+e_i, &\text{if } i\in M;\\
0, &\text{otherwise.}
\end{cases}
\end{equation}
Then $\supp(\tilde{e})=M$, $b+e=\tilde{b}+\tilde{e}$, $\|b+e\|_1=\|\tilde{b}\|_1+\|\tilde{e}\|_1$, and
\begin{equation}\label{E:maxa+b<maxa+maxb}
\max_{d\in\sigma P^*}d^{\top}(e+b)=\max_{d\in\sigma P^*}d^{\top}(\tilde{e}+\tilde{b})
\leq \max_{d\in\sigma P^*}d^{\top}\tilde{e}+\max_{d\in\sigma P^*}d^{\top}\tilde{b}.
\end{equation}
On one hand, it follows from Lemma~\ref{L:desigualdad-L1}\ref{L:L1i} with $y=\tilde{e}$, $g^*=0$, and $b^*=0$ that, for every 
$g\in\RR^p$, $\|\tilde{e}\|_1\leq\|\tilde{e}-Xg\|_1$, hence $0\in\argmin_{g\in\RR^p} \|\tilde{e}-Xg\|_1$ and from the first order optimality condition 
$0\in X^{\top}\partial\|\cdot\|_1(\tilde{e})$
or, equivalently,
\begin{equation*}
(\exists u\in P^*)\quad u^{\top}\tilde{e}=\|\tilde{e}\|_1.
\end{equation*}
Since, for every $u\in P^*$ , $u^{\top}e\leq \|e\|_1$ we hence deduce that
$\max_{u\in P^*}u^{\top}\tilde{e}=\|\tilde{e}\|_1$.
Therefore, by considering the change of variables $u=d/\sigma$, we obtain
\begin{equation}\label{E:dual-l1norm}
\max_{d\in\sigma P^*}d^{\top}\tilde{e}=\sigma\cdot\max_{u\in P^*}u^{\top}\tilde{e}=\sigma\|\tilde{e}\|_1.
\end{equation}
On the other hand, 
\begin{equation}\label{E:max-colinear}
\max_{d\in\sigma P^*}d^{\top}\tilde{b}\leq \max_{\|d\|_\infty\leq \sigma }d^{\top}\tilde{b}
=\frac{\sigma}{\|\tilde{b}\|_\infty}\tilde{b}^\top \tilde{b}=\frac{\sigma}{\|\tilde{b}\|_\infty}\|\tilde{b}\|_2^2.
\end{equation}
Therefore, by replacing \eqref{E:dual-l1norm} and \eqref{E:max-colinear} in \eqref{E:maxa+b<maxa+maxb}, the result follows from \eqref{e:deftilda}.
\end{proof}

\bibliographystyle{siam}
\bibliography{biblioeq}
\end{document}